\documentclass[11pt,a4paper,twoside,bibliography=totoc]{scrartcl}
\usepackage{a4wide}
\usepackage{amsfonts}
\usepackage{amsmath}
\usepackage{amssymb}
\usepackage{array}
\usepackage{amsthm}
\usepackage[utf8]{inputenc}
\usepackage{subcaption}
\usepackage{xifthen}
\usepackage{graphicx}
\usepackage{color}
\usepackage{pgf}
\usepackage{scrpage2}
\usepackage{multirow}
\usepackage{listings} 
\usepackage{tikz}
\usetikzlibrary{math}
\usepackage{graphicx}
\usepackage{multicol}

\newcommand{\N}{\ensuremath{\mathbb{N}}}
\newcommand{\T}{\ensuremath{\mathbb{T}}}

\newcommand{\Z}{\ensuremath{\mathbb{Z}}}

\newcommand{\R}{\ensuremath{\mathbb{R}}}
\newcommand{\C}{\ensuremath{\mathbb{C}}}

\newcommand{\nchoosek}[2]{\left(\begin{array}{c}#1\\#2\end{array}\right)}
\newcommand{\ii}{\textnormal{i}}
\newcommand{\e}{\textnormal{e}}

\newcommand{\eip}[1]{\textnormal{e}^{2\pi\ii{#1}}}
\newcommand{\eim}[1]{\textnormal{e}^{-2\pi\ii{#1}}}
\newcommand{\norm}[1]{\left\Vert #1\right\Vert}

\newcommand{\dotprod}[2]{ \left< #1,#2 \right>}

\newcommand{\dotprodltm}[2]{ \left< #1,#2 \right>_{L^2(\T^d)}}
\newcommand{\floor}[1]{\left\lfloor#1\right\rfloor}
\newcommand{\ceil}[1]{\left\lceil#1\right\rceil}

\newcommand{\pmat}[1]{\begin{pmatrix} #1 \end{pmatrix}}
\newcommand{\mat}[1]{\ensuremath \mathbf{\boldsymbol{#1}}}
\newcommand{\set}[1]{\left\{ #1 \right\}}

\newcommand{\interc}[1]{\left[ #1 \right]}
\newcommand{\intercl}[1]{\left[ #1 \right)}
\newcommand{\intercr}[1]{\left( #1 \right]}
\newcommand{\abs}[1]{\left| #1 \right |}
\newcommand{\wrap}[1]{\left| #1 \right|_{\mathbb{T}}}
\newcommand{\wrapm}[1]{\left| #1 \right|_{\mathbb{T}^d}}
\newcommand{\normlt}[1]{\left\Vert #1 \right\Vert_{L^2(\mathbb{T})}}
\newcommand{\normltm}[1]{\left\Vert #1 \right\Vert_{L^2(\mathbb{T}^d)}}
\newcommand{\normlinft}[1]{\left\Vert #1 \right\Vert_{L^\infty(\mathbb{T})}}
\newcommand{\normlinftm}[1]{\left\Vert #1 \right\Vert_{L^\infty(\mathbb{T}^d)}}
\newcommand{\conj}[1]{\overline{#1}}
\renewcommand{\vec}[1]{\ensuremath \mathbf{\boldsymbol{#1}}}
\newcommand{\valpha}{{\vec{\alpha}}}

\newcommand{\dirich}[1]{\ifthenelse{\isempty{#1}}
	{d_n}
	{d_n\left(#1\right)} }		
\newcommand{\dirichd}[1]{\ifthenelse{\isempty{#1}}
	{d_n'}
	{d_n'\left(#1\right)} }	
\newcommand{\dirichdd}[1]{\ifthenelse{\isempty{#1}}
	{d_n''}
	{d_n''\left(#1\right)} }	
\newcommand{\dirichm}[1]{\ifthenelse{\isempty{#1}}
	{\tilde d_n}
	{\tilde d_n\left(#1\right)} }
\newcommand{\comp}[2]{\left(#1\right)_{#2}}

\DeclareMathOperator*{\smin}{\sigma_{\min}}

\DeclareMathOperator*{\dist}{dist}

\newtheorem{thm}{Theorem}[section]
\newtheorem{lemma}[thm]{Lemma}
\newtheorem{remark}[thm]{Remark}
\newtheorem{definition}[thm]{Definition}
\newtheorem{example}[thm]{Example}
\newtheorem{corollary}[thm]{Corollary}
\newtheorem{proposition}[thm]{Proposition}

\numberwithin{equation}{section}
\numberwithin{table}{section}
\numberwithin{figure}{section}

\newcommand{\bend}{\hspace*{0ex} \hfill \hbox{\vrule height
    1.5ex\vbox{\hrule width 1.4ex \vskip 1.4ex\hrule  width 1.4ex}\vrule
    height 1.5ex}}

\long\def\symbolfootnote[#1]#2{\begingroup%
\def\thefootnote{\fnsymbol{footnote}}\footnote[#1]{#2}\endgroup}

\newcommand{\PP}{\mathcal{P}}

\newcommand{\CC}{\mathcal{C}}
\newcommand{\dd}{\,\mathrm{d}}

\newcommand{\OO}[1]{\mathcal{O}\left(#1\right)}

\renewcommand{\mathbf}[1]{\ensuremath{\boldsymbol{#1}}}

\newcommand{\rank}{ \operatorname{rank}}

\renewcommand{\thefootnote}{\fnsymbol{footnote}}

\allowdisplaybreaks
\title{On the smallest singular value of multivariate Vandermonde matrices with clustered nodes}

\date{\today}

\author{Stefan Kunis\footnotemark[1] \qquad Dominik Nagel\footnotemark[1]}

\newif\ifshow
\showtrue

\begin{document}
\maketitle

\begin{abstract}
	We prove lower bounds for the smallest singular value of rectangular, multivariate Vandermonde matrices with nodes on the complex unit circle.
	The nodes are ``off the grid'', groups of nodes cluster, and the studied minimal singular value is bounded below by the product of inverted distances of a node to all other nodes in the specific cluster.
	By providing also upper bounds for the smallest singular value, this completely settles the univariate case and pairs of nodes in the multivariate case, both including reasonable sharp constants.
	For larger clusters, we show that the smallest singular value depends also on the geometric configuration within a cluster.

	
	\medskip
	
\end{abstract}

\footnotetext[1]{
	Osnabr\"uck University, Institute of Mathematics
	\texttt{\{skunis,dnagel\}@uos.de}
}


\section{Introduction}  
Vandermonde matrices appear e.g.~in the stability analysis of super-resolution algorithms like Prony's method \cite{Pr95,KuMoPeOh17}, the matrix pencil method \cite{HuSa90,Mo15}, the ESPRIT algorithm \cite{RoKa89,PoTa2017,LiLiFa19}, and the MUSIC algorithm \cite{Sc86,LiFan2016}.
We are interested in the case of nodes on the complex unit circle and a large polynomial degree, the matrices then generalize the classical discrete Fourier matrices to non-equispaced nodes and the involved polynomial degree is also called bandwidth.
If all nodes are well-separated, bounds on the condition number are established for example in \cite{Ba99,KuPo07,Mo15,AuBo17,Di19} for the univariate case and in \cite{KuPo07,KuMoPeOh17} at least partially for the multivariate case.
For node sets with distances of which some are below the inverse bandwidth, the behavior of the smallest singular value is subject of current research.
The seminal paper \cite{Do92} coined the term (inverse) super-resolution factor for the product of the bandwidth and the minimal separation of the nodes.
For $M$ nodes on a grid, the results in \cite{Do92,DeNg15} imply that the smallest singular value is at most as small as the inverse super-resolution factor raised to the power of $M-1$ if the super-resolution factor is greater than $1$.
More recently, the practically relevant situation of clustered nodes was studied in \cite{CaMo16,GoYo17,LiLi17,BaDeGoYo18,KuNa18,BaGoYo19,Di19}.
In the univariate case and for different setups, all of these refinements are able to replace the exponent $M-1$ by the smaller number $m-1$, where $m$ denotes the number of nodes that are in the largest cluster of nodes.

Here, we refine the proof technique developed in the second version of \cite{LiLi17} and extend it to arbitrary dimensions. In contrast to \cite{LiLi17}, we only use the information on the biggest cluster size, minimal separation between clusters and a the worst case cluster complexity (or a minimal separation between nodes) instead of taking the structure of each cluster into account.
In summary, our contributions are:
\begin{enumerate}
 \item a refined analysis of the univariate case, cf.~\cite{LiLi17}, eliminating the dependence on the total number of nodes, weakening a technical condition on the cluster separation, and improving constants, mainly by
  \begin{enumerate}
   \item a geometric packing argument and
   \item an improved estimate of Dirichlet kernels and Lagrange-like basis functions;
  \end{enumerate}
 \item a multidimensional generalization, including
 \begin{enumerate}
  \item a quantitative estimate for the well-separated case,
  \item a sharp estimate for pair clusters in higher dimensions, and
  \item an example on the limitations for larger clusters in higher dimensions.
 \end{enumerate}
\end{enumerate}
The outline of this paper is as follows:
Section 2 fixes notation, states the problem and gives some definitions. Furthermore, we generalize the so-called robust duality lemma from the second version of \cite{LiLi17} to the multivariate case.
In Section 3, we introduce some auxiliary functions which are used to prove our main results in Section 4.
Additionally, we give examples with specified parameters, present implications of our result for special node configurations like pair clusters and well separated nodes, and compare them with existing results. 
In Section 5, upper bounds on the smallest singular value for the univariate case and for pair clusters in higher dimensions are presented - these match the lower bounds from our main theorem.
Furthermore, an example of a triple cluster in two dimensions is given which shows that geometric properties beyond pairwise distances are needed for understanding the multivariate case.
Finally, in Section 6 numerical experiments are presented that support statements and comparisons from preceding sections.

\section{Preliminaries}
\begin{definition}[Setting]\label{def:setting}
	We denote the component of a vector by bracketing and setting a subscript, unless its components are defined differently.
	Let $d\in \N$ be a given dimension and $\Omega:=\set{\vec{t}_1,\dots,\vec{t}_M} \subset \intercl{0,1}^d$ a set of points.
	The corresponding nodes are given by $\vec{z}_j:= \eip{\vec{t}_j}\in \T^d, j=1,\dots,M$, where $\T:=\set{z\in\C \colon \abs{z}=1}$ denotes the complex unit circle.
	We identify the unit interval with the unit circle and therefore, we do not make a difference between the $\vec{t}_j$ and $\vec{z}_j$ and call them both nodes. Throughout the paper, $\norm{\cdot}_2$ denotes the euclidean norm for vectors and also its induced norm for matrices, and analogously $\norm{\cdot}_\infty$ the max-norm.
	Let $n\in \N$ be a degree, set $N:=n+1$ and assume $M<N^d$. 
	We are interested in the multivariate, rectangular Vandermonde matrix
	\begin{equation}\label{eq:defA}
		\mat{A}	
		:=\mat{A}(\Omega,n)
		:= \pmat{\vec{z}_j^\valpha}_{\substack{ j=1,\dots,M \\ \valpha\in\N_0^d,\,\norm{\valpha}_\infty \le n}}\in \C^{M\times N^d},
	\end{equation}
	and its smallest singular value
	\begin{equation*}
		 \smin(\mat{A}) := \min_{\substack{\vec{v}\in \C^{M}\\ \norm{\vec{v}}_2=1}}\norm{\mat{A}^*\vec{v}}_2.
	\end{equation*} 
\end{definition}

The following lemma builds the core of the proof technique developed in the second version of \cite{LiLi17} which we adapt here to the multivariate setting.
\begin{lemma}[{Robust duality, cf.~\cite[v2, Prop.~2]{LiLi17}}]\label{la:robustDual}
	Let $\Omega$ and $\mat{A}$ be given as in Definition \ref{def:setting}.
	If for any unit norm vector $\vec{v}=(v_1,\dots,v_M)^\top\in \C^M$, and $\vec{\epsilon}=(\epsilon_1,\dots,\epsilon_M)^\top\in \C^M with \norm{\vec{\epsilon}}_2\le 1$,
	there exists a trigonometric polynomial of max-degree at most $n\in \N$, i.e.,
	\begin{equation*}
	 f\in \PP(n):=\set{ g\colon \T^d \rightarrow \C \colon g(\vec{t}) 
		=\sum_{\vec{\alpha}\in \N_0^d, \norm{\valpha}_\infty\le n} \hat g_\valpha \eip{\valpha\cdot \vec{t}}, \hat g_\valpha \in \C},
	\end{equation*}
	such that $f(\vec{t}_j)=v_j+\epsilon_j$ for each $j=1,\dots,M$, then
	\begin{equation*}
		\norm{\mat{A}^*\vec{v}}_2 \ge (1-\norm{\vec{\epsilon}}_2) \norm{f}_{L^2(\T^d)}^{-1}.
	\end{equation*}
\end{lemma}
\begin{proof}
	Define the discrete measure $\mu := \sum_{j=1}^{M} v_j \delta_{\vec{t}_j}$. Its Fourier coefficients are given by
	\begin{equation*}
		\hat \mu (\valpha) 
		= \int_{\T^d} \eim{\vec{t}\cdot \vec{\valpha}} \dd \mu(\vec{t})
		= \sum_{j=1}^{M}v_j \vec{z}_j^{-\valpha} 
		= \comp{\mat{A}^*\vec{v}}{\valpha},\quad \valpha\in\N_0^d,\norm{\valpha}_{\infty}\le n.
	\end{equation*}
	On the one hand, using the interpolation property of $f$ and the lower triangular inequality of the absolute value, we have 
	\begin{equation*}
		\abs{\int_{\T^d} \conj{f} \dd \mu} 
		= \abs{ \sum_{j=1}^{M}\conj{f(\vec{t}_j)} v_j } 
		= \abs{ \norm{\vec{v}}_2^2 + \sum_{j=1}^{M} \conj{\epsilon_j} v_j}
		\ge \norm{\vec{v}}_2^2 - \norm{\vec{v}}_2\norm{\vec{\epsilon}}_2
		= (1-\norm{\vec{\epsilon}}_2),
	\end{equation*}
	and on the other hand, using $f\in\PP(n)$, the Cauchy--Schwarz inequality and Parseval's identity, we have
	\begin{equation*}
		\abs{\int_{\T^d}\conj f \dd\mu}
		= \abs{\sum_{\valpha\in\N_0^d,\norm{\valpha}_{\infty} \le n}\conj{\hat f_\valpha} \hat \mu(\valpha) }
		\le \norm{\vec{\hat f}}_2\norm{\mat{A}^* \vec{v}}_2 
		= \norm{f}_{L^2(\T^d)} \norm{\mat{A}^* \vec{v}}_2.
	\end{equation*}
\end{proof}
The advantage of that lemma is, if $\vec{v}\in \C^M$ is a unit norm vector such that $\smin(\mat{A})=\norm{\mat{A}^*\vec{v}}_2$, it suffices to construct a function $f\in\PP(n)$ almost interpolating the values of $\vec{v}$ in order to provide a lower bound.

The following definition is similar to the `localized clumps' model from the second version of \cite{LiLi17}. We did some renaming in terms of \cite{BaDeGoYo18} and use a normalization by $N$ rather than $n$.
\begin{definition}[Geometry of nodes]\label{def:geom}
	The wrap-around distance between two nodes $\vec{t}, \vec{t}' \in [0,1)^d$ is defined by
	\begin{equation*}
		\wrapm{\vec{t}-\vec{t}'}:= \min_{\vec{r}\in \Z^d} \norm{\vec{t}-\vec{t}'+\vec{r}}_\infty.
	\end{equation*}
	\begin{enumerate}
	\item A subset of nodes is called \emph{cluster} if it is contained in a cube of length $1/N$. For two clusters $\Lambda',\Lambda'' \subset \Omega$, we define
	\begin{equation*}
	 \dist(\Lambda',\Lambda''):=\min\{\wrapm{\vec{t}'-\vec{t}''}:\;\vec{t}'\in\Lambda',\;\vec{t}''\in\Lambda''\}. 
	\end{equation*}
	\item The node set $\Omega$ is called a \emph{clustered node configuration} with $L$ clusters if it can be written as 
	\begin{equation*}
		\Omega = \bigcup_{l=1}^{L} \Lambda_l,
	\end{equation*}
	where the $\Lambda_l$ are clusters and the \emph{(normalized) minimal cluster separation} $\rho$ fulfills
	\begin{equation*}
		\rho:=N\min_{1\le l < l' \le L}\dist(\Lambda_{l},\Lambda_{l'}) > 1.
	\end{equation*} 
	We order $\abs{\Lambda_1}\ge\abs{\Lambda_2}\ge\hdots\ge\abs{\Lambda_L}$ and denote the cardinality of the biggest cluster by $\lambda:=\abs{\Lambda_1}$.
	In passing, we note that the node set $\Omega$ is called \emph{well separated} with normalized separation $\rho$ if $\lambda=1$.
	Moreover, we define the partitioning of $\T^d$ into \emph{shells} by
	\begin{equation*}
	J_m:=J_m(\Omega,N,\rho)
	:= \set{\vec{t} \in \T^d \colon m\rho\le N\wrapm{\vec{t}}<(m+1)\rho}, \qquad m=0,\hdots,\floor{\frac{N}{2\rho}}.
	\end{equation*}
	\item The \emph{cluster complexity} is defined by
	\begin{equation*}
		\CC
		:=\CC(\Omega,N)
		:= \max_{j=1,\hdots,M} \prod_{\vec{t}' \in \Omega \colon 0<
			 \wrapm{\vec{t}_j-\vec{t}'}\le 1/N} \frac{1}{N\wrapm{\vec{t}_j-\vec{t}'}}
	\end{equation*}
	and finally, we define the \emph{(normalized) minimal separation}
	\begin{equation*}
	 \tau:=N\min_{1\le j < j' \le M} \wrapm{\vec{t}_j-\vec{t}_{j'}}.
	\end{equation*}
	\end{enumerate}
\end{definition}

\begin{figure}[h]
\begin{center}
	\begin{subfigure}{0.32\textwidth}
	    \includegraphics[width=1.0\linewidth]{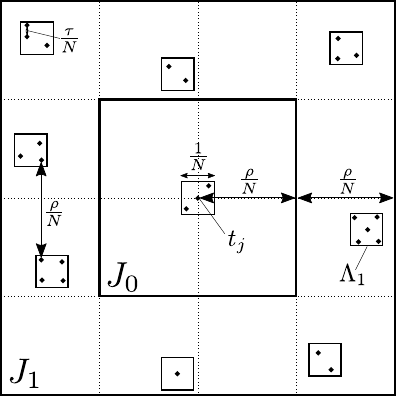}
	\end{subfigure}
	\hspace{1cm}
	\begin{subfigure}{0.4\textwidth}
		\includegraphics[width=1.0\linewidth]{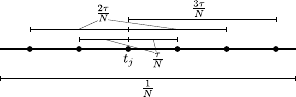}
	\end{subfigure}
	\caption{Left: Clustered node configuration and partitioning into shells for $d=2$; right: a cluster which maximizes $\CC$ for $d=1$.}
	\label{fig:cluster2d1d}
\end{center}
\end{figure}

\begin{remark}(Geometry of nodes)\label{rm:nodeGeometry}
 With the notation of Definition \ref{def:geom}, we note that
 \begin{enumerate}
  \item the inequality $\sin(x) \ge 2x/\pi $ for $0\le x \le \pi/2$ implies
  \begin{equation}\label{eq:node_distance}
		\abs{z - z'}= 2 \sin(\pi \wrap{t - t'})\ge 4 \wrap{t - t'}, \qquad z:=\eip{t},\,z':=\eip{t'} \in \T.
	\end{equation}
	A higher order approximation is given in the second version of \cite{LiLi17},
	\begin{equation}\label{eq:node_distance_highO}
		\abs{z - z'} 
		\ge 2\pi \left(1-\frac{\pi^2 \wrap{t-t'}^2}{3}\right)^{1/2} \wrap{t-t'}.
	\end{equation}
  \item A necessary condition on $N$ for the existence of a clustered node configuration with $L$ clusters is $L\rho^d\le N^d$, with equality if and only if all nodes are equispaced.
  Similarly, if $N\ge L^{1/d}(\rho+1)$, then equispaced cluster with arbitrary node configuration within each cluster exist.
  Moreover, the cluster separation $\rho$ needs to scale at least linearly in the biggest cluster size $\lambda$.
  If on the contrary, $\rho < \lambda/4$ and $d=1$ for simplicity of the argument, then let $\lambda$ nodes form a cluster (length at most $1/N$) and place one node as far as possible away.
  With fixed $N$, we have $2\rho = N-1$ and therefore, $\rho<\lambda/4$ is equivalent to $N\le M/2+1/2$ and thus $\rank(\mat{A})\le N<M$.
  On the other hand, $\rho > \lambda$ already implies $N\ge L(\rho+1) \ge M$.
	
 Finally note, that the packing argument in \cite[Lemma 4.5]{KuPo07} yields
	\begin{equation*}
	\abs{J_m\cap \Omega}
	\le 2^d\left(2^d-1\right)m^{d-1}\lambda,
	\end{equation*}
	see also Figure \ref{fig:cluster2d1d} (left).

%
%
%


  \item The cluster complexity can be upper bounded by the normalized minimal separation as follows.
  For $d\in\N$, we have $\CC \le \tau^{1-\lambda}$ and equality for $\lambda=1$ and $\lambda=2$. Refined for $d=1$, it is easy to see that the cluster complexity is maximized by an equispaced cluster with $\lambda$ nodes separated by $\tau/N$ and taking distances from the center node, see Figure \ref{fig:cluster2d1d} (right).
  By logarithmic convexity, direct calculation, and Stirling's approximation, we thus have
	\begin{equation}\label{eq:CupperBound}
		\CC 
		\le \frac{1}{\tau^{\lambda-1}} \left(\floor{\frac{\lambda-1}{2}}! \cdot \ceil{\frac{\lambda-1}{2}}!\right)^{-1}
		\le \frac{1}{\tau^{\lambda-1}\Gamma\left(\frac{\lambda+1}{2}\right)^{2}}
		\le \frac{(2\e)^{\lambda-1}}{\lambda^{\lambda}} \cdot \frac{1}{\tau^{\lambda-1}}
	\end{equation}
	and similarly
	\begin{equation}\label{eq:ClowerBound}
		\max_{\Omega}\CC 
		= \frac{1}{\tau^{\lambda-1}} \left(\floor{\frac{\lambda-1}{2}}! \cdot \ceil{\frac{\lambda-1}{2}}!\right)^{-1}
		\ge \frac{(2\e)^{\lambda-1}}{\lambda^{\lambda+1}} \cdot \frac{1}{\tau^{\lambda-1}},
	\end{equation}
	where the maximum is taken over all clustered node configurations with normalized minimal separation $\tau$ and the largest cluster containing $\lambda$ nodes.
 \end{enumerate}
\end{remark}

\section{Auxiliary functions}

\begin{lemma}[Modified Dirichlet kernel]\label{la:modDirich}
	For $m, \beta \in \N$ the modified Dirichlet kernel is defined as $d_m\colon \intercl{0,1}\rightarrow \C$,
	\begin{equation*}\label{eq:modDirich}
		d_m(t)
		:= \frac{1}{m+1} \sum_{k=0}^{m} \eip{kt} 
		= 	\begin{cases}
				1,& t=0,\\
				\frac{\e^{\pi \ii m t}}{m+1}\cdot \frac{\sin(\pi(m+1)t)}{\left(\sin(\pi t)\right)},& t\ne 0.
			\end{cases}
	\end{equation*}
	We define the powers of the multivariate modified Dirichlet kernel by
	\begin{equation*}
		d_m^\beta\colon \intercl{0,1}^d\rightarrow \C,
		\quad d_m^\beta(\vec{t}) := \left(\prod_{\ell=1}^{d} d_m(\comp{\vec{t}}{\ell})\right)^\beta \in \PP(m \beta).
	\end{equation*}
	If $m\ge \beta$ and $\vec{t} \in \T^d\setminus \set{\vec{0}}$, then
	\begin{enumerate}
		\item $\abs{d_m(\vec{t})}\le d_m(\vec{0}) = 1$,
		\item $\abs{d_m(\vec{t})}\le \frac{1}{2(m+1) \wrapm{\vec{t}}}$,
		\item $\norm{d_m^\beta}_{L^2(\T^d)}^2 \le \frac{1}{(m+1)^d \beta^{d/2}}$,
		\item $\abs{\dotprod{d_m^\beta}{d_m^\beta(\cdot - \vec{t})}_{L^2(\T^d)}} \le \frac{1}{2(m+1)^d \beta^{(d-1)/2}}\cdot \frac{1}{(m+1)^{\beta}\wrapm{\vec{t}}^{\beta}}$.
	\end{enumerate}
\end{lemma}
\begin{proof}
 First, note that
	\begin{equation*}
		\abs{d_m(\vec{t})}
		\le  \left(\frac{1}{m+1}\sum_{k=0}^{m} \abs{\eip{kt}}\right)^d
		= 1
		= d_m(\vec{0})
	\end{equation*}
	and the point-wise bound follows in the univariate case by
	\begin{equation*}
		\abs{d_m(t)}
		= \frac{1}{m+1} \abs{\frac{\sin(\pi(m+1)t)}{\sin(\pi t)}}
		\le \frac{1}{(m+1)\abs{\sin(\pi t)}}
		\le \frac{1}{2(m+1)\wrap{t}}.
	\end{equation*}
	
	Second, in the multivariate case, setting $t:=\wrapm{\vec{t}}$, and using i) and the univariate bound yield
	\begin{equation*}
		\abs{d_m(\vec{t})}
		= \prod_{\ell=1}^{d} \abs{d_m(\comp{\vec{t}}{\ell})}
		\le \abs{d_m(t)}
		\le \frac{1}{2(m+1)\wrap{t}}
		=\frac{1}{2(m+1)\wrapm{\vec{t}}}.
	\end{equation*}
	
	Note that $\normltm{d_m}^2 = \normlt{d_m}^{2d}$ and therefore, the third assertion is proven for the univariate case as follows. For $m\ge \beta$, Parseval's identity and direct calculation show 
	\begin{align*}
		\normlt{d_m}^2&=\frac{1}{m+1},\quad
		\normlt{d_m^2}^2=\frac{1}{m+1}\left[\frac{2}{3}+\frac{1}{3(m+1)^2}\right]\le \frac{1}{m+1}\cdot\frac{19}{27} \le \frac{1}{m+1}\cdot\frac{1}{\sqrt{2}},\\
		\normlt{d_m^3}^2&=\frac{1}{m+1}\left[\frac{11}{20}+\frac{1}{4(m+1)^2}+\frac{1}{5(m+1)^4}\right]\le\frac{1}{m+1} \cdot \frac{145}{256} \le \frac{1}{m+1}\cdot\frac{1}{\sqrt{3}}.
	\end{align*}
	For $x\in[0,1]$ and $m\ge 4$, the estimates in \cite[Proof of Lemma 2]{MaRo08} yield
	\begin{align*}
		\frac{\sin \pi x}{(m+1)\sin\frac{\pi}{m+1}x}
		&\le \exp\left(-\frac{\pi^2((m+1)^2-1)}{6(m+1)^2}x^2\right)
		\le \exp\left(-\frac{4\pi^2 x^2}{25}\right)
	\end{align*}
	and thus, for $m\ge\beta\ge 4$, the remaining estimate
	\begin{align*}
		\normlt{d_m^\beta}^2
		&= \frac{2}{(m+1)^{2\beta}} \int_{0}^{1/2} \abs{\frac{\sin(\pi(m+1)t)}{\sin(\pi t)}}^{2\beta} \dd t \\
		&= \frac{2}{m+1} \left[ \frac{1}{(m+1)^{2\beta}} \left(\int_{0}^{1} \left(\frac{\sin(\pi x)}{\sin(\frac{\pi}{m+1}x)}\right)^{2\beta} \dd x + \int_{1}^{\frac{m+1}{2}} \abs{\frac{\sin(\pi x)}{\sin(\frac{\pi}{m+1}x)}}^{2\beta}  \dd x\right)\right] \\
		&\le \frac{2}{m+1} \left[\int_{0}^{\infty} \exp\left(-\frac{8\beta\pi^2 x^2}{25}\right) \dd x + \int_{1}^{\infty} \left(\frac{1}{2x}\right)^{2\beta} \dd x\right]\\
		&= \frac{1}{m+1} \left[\frac{5}{2\sqrt{2\pi}} \frac{1}{\sqrt{\beta}}  + \frac{2^{1-2\beta}}{2\beta-1}\right] 
		\le \frac{1}{m+1} \cdot \frac{1}{\sqrt{\beta}}.
	\end{align*}
	
	In order to prove the fourth assertion, note $\wrap{t}\le\wrap{t-t'}+\wrap{t'}\le 2\max\{\wrap{t-t'},\wrap{t'}\}$ and hence, i) and ii) yield
	\begin{equation*}
		\abs{d_m(t')}\abs{d_m(t'-t)} 
		\le\frac{1}{2(m+1)}\min\left\{\frac{1}{\wrap{t'-t}},\frac{1}{\wrap{t'}}\right\}
		\le \frac{1}{(m+1)\wrap{t}}
	\end{equation*}
	and $\normlinft{d_m d_m(\cdot - t)} \le ((m+1)\wrap{t})^{-1}$.
	Moreover, direct computation gives 
	\begin{equation*}
		 \abs{\dotprod{d_m}{d_m(\cdot - t)}_{L^2(\T)}}=\frac{\abs{d_m(t)}}{m+1}\le\frac{1}{m+1}\cdot\frac{1}{2}\cdot\frac{1}{(m+1)\wrap{t}}.
	\end{equation*}
	and with $z=\eip{t}$ and Parseval's identity also
	\begin{align*}
	  \abs{\dotprod{d_m^2}{d_m^2(\cdot - t)}_{L^2(\T)}}
	  &=\frac{1}{(m+1)^4}\int_{\T} \left(\frac{\sin\pi(m+1)t'}{\sin\pi t'} \cdot\frac{\sin\pi(m+1)(t'-t)}{\sin\pi (t'-t)}\right)^2 \dd t'\\
	  &=\abs{\sum_{k=-m}^m (m+1-\abs{k})^2 z^k}\\
	  &=\frac{\abs{(z+1)z\left(z^{m+1}-z^{-m-1}\right) +4(m+1)z(1-z)}}{(m+1)^4 \abs{z-1}^3}\\
	  &\le\frac{1}{m+1}\cdot\frac{1}{2}\cdot\frac{1}{(m+1)^2\wrap{t}^2}.
	\end{align*}
	Finally, let $t$ be the coordinate with $\wrap{t}=\wrapm{\vec{t}}$, then the Cauchy--Schwarz inequality, iii), and the above yield
	(noting that $\e^{-2\pi \ii m t'} d_m^2(t')\ge 0$ and omitting the second last line if $\beta=1$)
	\begin{align*}
		\abs{\dotprod{d_m^\beta}{d_m^\beta(\cdot - \vec{t})}_{L^2(\T^d)}}
		&\le \abs{\int_{\T^d} d_m^\beta(\vec{t}') \overline{d_m^\beta(\vec{t}'-\vec{t})} \dd \vec{t}'}\\
		&\le \normlt{d_m^\beta}^{2(d-1)} \abs{\int_{\T} d_m^\beta(t') \overline{d_m^\beta(t'-t)} \dd t'}\\
		&\le \normlt{d_m^\beta}^{2(d-1)} \normlinft{d_m d_m(\cdot-t)}^{\beta-2} \abs{\dotprod{d_m^2}{d_m^2(\cdot - t)}_{L^2(\T)}}\\
		&\le \frac{1}{2(m+1)^d \beta^{(d-1)/2}} \cdot \frac{1}{(m+1)^{\beta} \wrap{t}^{\beta}}.
	\end{align*}

\end{proof}

\begin{lemma}[Lagrange-like basis with decay, cf.~{\cite[v2, Lem.~3]{LiLi17}}]\label{la:lagrangeBasisDecay}
	Let $\beta,d,M,n\in \N$, $\beta$ be even, $\Omega = \set{\vec{t}_1,\hdots,\vec{t}_M} \subset [0,1)^d$ be a clustered node configuration and $n\ge 2\beta^2 \lambda$.
	Then for each $\vec{t}_j\in \Omega$ with $\vec{t}_j \in \Lambda_l$ for some $l=l(j)$, there exists an $I_j \in \PP(n)$, such that 
	\begin{enumerate}
		\item $I_j(\vec{t}_k)=\delta_{j k}$ for all $\vec{t}_k \in \Lambda_l$,
		\item $\abs{I_j(\vec{t})}
		\le  \frac{\beta^{\beta} \lambda^{\beta+\lambda-1}}{(2N\wrapm{\vec{t}-\vec{t}_j})^{\beta}} \CC$ for all $\vec{t}\neq \vec{t}_j$, and
		\item $\abs{\dotprodltm{I_k}{I_j}} 
		\le \frac{\lambda^d\beta^{d/2}}{N^d}\lambda^{2\lambda-2}\CC^2
		 \begin{cases}
				1,&\vec{t}_k \in \Lambda_l, \\ \frac{\sqrt{\beta}}{2}\left(\frac{\lambda\beta}{N\wrapm{\vec{t}_j-\vec{t}_k}}\right)^{\beta}, &\text{otherwise.}		
\end{cases}$
	\end{enumerate}
\end{lemma}
\begin{proof}
 We define the functions $I_j$ as product of a Lagrange polynomial $G_j$ within the cluster and a fast decaying function $H_j$.
 Let $j\in\{1,\hdots,M\}$ be fixed and define the $j$-th Lagrange polynomial within its cluster $\Lambda_l$, $l=l(j)$, as follows.
 If $\abs{\Lambda_l}=1$, we simply set $G_j\equiv 1$. Otherwise, let 
 \begin{equation}\label{eq:defQ}
  Q := \floor{\frac{n}{\lambda}}\ge \frac{n-\lambda+1}{\lambda} \ge \frac{N}{2\lambda}
 \end{equation}
 denote the 'blow-up-factor' and for
 $\vec{t}_k\in\Lambda_l\setminus \set{\vec{t}_j}$ let $\ell(k)$ be the index of the vector component that realizes the distance $\wrapm{\vec{t}_j -\vec{t}_k}$.
 We immediately have $\wrapm{Q\vec{t}_j-Q\vec{t}_k}=Q\wrapm{\vec{t}_j-\vec{t}_k}\ne 0$ and thus
 \begin{equation*}
  G_j(\vec{t}) := \prod_{\vec{t}_k\in \Lambda_l\setminus \set{\vec{t}_j}} \frac{\eip{Q\comp{\vec{t}}{\ell(k)} } - \eip{Q \comp{\vec{t}_k}{\ell(k)}}}{\eip{Q\comp{\vec{t}_j}{\ell(k)} } - \eip{Q \comp{\vec{t}_k}{\ell(k)}}}
 \end{equation*}
 fulfills $G_j(\vec{t}_k)=\delta_{j,k}$ and by inequality \eqref{eq:node_distance}
 \begin{equation}\label{eq:GBound}
    \normlinftm{G_j}
    \le \prod_{\vec{t}_k\in \Lambda_l\setminus \set{\vec{t}_j}} \frac{1}{2Q \wrap{\comp{\vec{t}_j}{\ell(k)}-\comp{\vec{t}_k}{\ell(k)}}}
    \le  \lambda^{\lambda-1} \CC.
 \end{equation}
 
 We proceed by setting
 \begin{equation*}
  P:= \floor{\frac{n}{\lambda\beta}},\quad P+1\ge\frac{n-\lambda\beta+1}{\lambda\beta}+1\ge\frac{N}{\lambda\beta},
 \end{equation*}
 and $H_j(\vec{t}):= d_P^\beta(\vec{t}-\vec{t}_j)$.
 Lemma \ref{la:modDirich} yields $H_j(\vec{t}_j)=1$,
 \begin{align*}
 \abs{H_j(\vec{t})} &\le \left(\frac{1}{2(P+1)\wrapm{\vec{t}-\vec{t}_j}}\right)^{\beta}	\le \left(\frac{\lambda\beta}{2N\wrapm{\vec{t}-\vec{t}_j}}\right)^{\beta},\quad &\vec{t} \neq \vec{t}_j,\\
 \abs{\dotprod{H_k}{H_j}_{L^2(\T^d)}} & \le \normltm{d_{P}^\beta}^2 \le \frac{1}{(P+1)^{d}\beta^{d/2}} \le \frac{\lambda^d\beta^{d/2}}{N^d},\quad &\vec{t}_k \in \Lambda_l,\\
 \abs{\dotprod{H_k}{H_j}_{L^2(\T^d)}} &
  =\abs{\dotprod{d_P^\beta}{d_P^\beta(\cdot-(\vec{t}_j-\vec{t}_k)}_{L^2(\T^d)}}
  \le \frac{\lambda^{d}\beta^{(d+1)/2} (\lambda\beta)^{\beta}} {2N^d(N\wrapm{\vec{t}_j-\vec{t}_k})^{\beta}},\quad &\vec{t}_k \notin \Lambda_l.
 \end{align*}
 
 Finally, we define $I_j(\vec{t}) := G_j(\vec{t}) H_j(\vec{t})$.
 This yields $I_j \in \PP(n)$ since $G_j\in\PP(Q(\lambda-1))$, $H_j \in \PP(P\beta)$, and
 \begin{equation*}
  P\beta + (\lambda-1) Q \le \frac{n}{\lambda}+(\lambda-1)\frac{n}{\lambda}= n.
 \end{equation*} 
 Moreover, this function has the desired property $I_j(\vec{t}_k)=\delta_{jk}$ for all $\vec{t}_k\in \Lambda_l$ and
 the two remaining inequalities follow by $\abs{I_j(\vec{t})} \le \normlinftm{G_j} \abs{H_j(\vec{t})}$
 and by using $\e^{-\pi \ii \beta P (\vec{t}-\vec{t}_j)} H_j(t) \ge 0$, also $\abs{\dotprod{I_k}{I_j}_{L^2(\T^d)}}\le \norm{G_j}_{L^\infty(\T^d)}^2 \abs{\dotprod{H_k}{H_j}_{L^2(\T^d)}}$.
\end{proof}
\begin{remark}\label{rm:LiLiEstimate}
  Following the calculation in the second version of \cite[p.~36]{LiLi17}, we can improve \eqref{eq:GBound} to
   \begin{equation*}
    \normlinftm{G_j} \le \left(1-\frac{\pi^2}{3\lambda^2}\right)^{\frac{1-\lambda}{2}} \left( \frac{N}{\lambda} / \floor{\frac{n}{\lambda}}\right)^{\lambda-1} \left(\frac{\lambda}{\pi}\right)^{\lambda-1} \CC
                     \le 2.4\left(\frac{C(n) \lambda}{\pi}\right)^{\lambda-1} \CC
   \end{equation*}
   with $C(n)\to 1$ for $n\to\infty$ and where the first two bracketed terms are due to \eqref{eq:node_distance_highO} and \eqref{eq:defQ}, respectively.
\end{remark}

\section{A lower bound on the smallest singular value}

In this chapter we work out the multivariate extension of Theorem 1 in the second version of \cite{LiLi17}. Additionally, we do an improvement on the cluster separation condition, especially make the cluster separation independent on the number of nodes $M$. Furthermore, we provide an improved estimate on the smallest singular value $\smin(\mat{A})$ only depending on the biggest cluster size $\lambda$ and not on the number of all nodes $M$.

\begin{thm}\label{thm:lowerBoundSminImproved}
	Let $\beta, d, N, M\in \N$, $\beta \ge d+1$ even, $\Omega = \set{\vec{t}_1,\hdots,\vec{t}_M} \subset [0,1)^d$ be a clustered node configuration and $N> 2\beta^2\lambda$.
	Moreover, assume the cluster separation
	\begin{equation}\label{eq:clusterSepThm}
	\rho
	\ge \lambda \beta\left(\beta^{1/2}2^{d}(2^d-1)\lambda^{\lambda} \zeta(\beta-d+1) \CC\right)^{\tfrac{1}{\beta}}.
	\end{equation}
	Then the smallest singular value of the Vandermonde matrix $\mat{A}\in\C^{M\times N^d}$ from Definition \ref{def:setting} is bounded by
	\begin{equation*}
	\smin(\mat{A}) 
	\ge \left(1.5\cdot \beta^{d/4}\lambda^{\lambda+d/2-1/2}\right)^{-1} \frac{N^{d/2}}{\CC}.
	\end{equation*}
\end{thm}
\begin{proof}
	We apply the robust duality from Lemma \ref{la:robustDual},
	with $\vec{v}\in\C^M$, $\norm{\vec{v}}_2=1$, such that $\smin(\mat{A}) = \norm{\mat{A}^*\vec{v}}_2$, and
	\begin{equation*}
		f:= \sum_{k=1}^{M} v_k I_k,
	\end{equation*}
	where the Lagrange-like basis functions $I_k$ are given by Lemma \ref{la:lagrangeBasisDecay}.
	The interpolation errors $\epsilon_j=f(\vec{t}_j)-v_j$ fulfill $\vec{\epsilon}=\mat{K}\vec{v}$, where $\mat{K}\in\C^{M\times M}$ has the entries
	\begin{equation*}
		K_{j,k}	:= 	\begin{cases}
				0,&\quad j=k,\\
				I_k(t_j),&\quad j\ne k.
			\end{cases}
	\end{equation*}
	We proceed by $\norm{\vec{\epsilon}}_2 \le \norm{\mat{K}}_2\le \norm{\mat{\tilde K}}_2$, where the second inequality follows from monotonicity of the norm \cite[p.~520]{HoJo13} (or \cite[Lem.~A.2]{KuNa18}) and
	Lemma \ref{la:lagrangeBasisDecay} i) and ii) with
	\begin{equation*}
		\tilde K_{j,k} :=
		    \begin{cases}
			0,& \vec{t}_j\in\Lambda_{l(k)},\\
			\frac{\beta^{\beta} \lambda^{\beta+\lambda-1}}{(2N\wrapm{\vec{t}_k-\vec{t}_j})^{\beta}} \CC, & \text{otherwise}.
			\end{cases}
	\end{equation*}
	Since $\mat{\tilde K}\in \R^{M\times M}$ is symmetric, we bound the spectral norm by the maximum norm and apply the packing argument from Definition \ref{def:geom} ii) and Remark \ref{rm:nodeGeometry} ii) to get
	\begin{align*}
		\norm{\vec{\epsilon}}_2
		&\le\max_{j=1,\hdots,M} \sum_{\substack{k=1 \\ k\ne j}}^{M} \tilde K_{j,k}
		\le \lambda 2^d(2^d-1) \sum_{m=1}^{\floor{N/2\rho}} m^{d-1} \max_{\vec{t}\in J_m} \frac{\beta^{\beta} \lambda^{\beta+\lambda-1}}{(2N\wrapm{\vec{t}})^{\beta}} \CC\\
		&\le 2^{d-\beta}(2^d-1) \lambda^{\beta+\lambda} \beta^{\beta}\CC \zeta(\beta-d+1) \rho^{-\beta}.
	\end{align*}
	Condition \eqref{eq:clusterSepThm} and $\beta \ge 2$ imply $\norm{\vec{\epsilon}}_2\le\frac{1}{4\sqrt{2}}$.
	To bound the $L^2$-norm of $f$, let ${\mat{\hat K}} := \pmat{\abs{\dotprod{I_k}{I_j}}}_{j,k=1,\hdots,M} \in \R^{M\times M}$.
	The triangle inequality, symmetry of ${\mat{\hat K}}$, Lemma \ref{la:lagrangeBasisDecay} iii), and the packing argument from Definition \ref{def:geom} ii) and Remark \ref{rm:nodeGeometry} ii) yield
	\begin{align*}
		\normltm{f}^2
		&= \sum_{j=1}^{M}\sum_{k=1}^{M} v_k \conj{v}_j \dotprod{I_k}{I_j}_{L^2(\T^d)}
		\le  \max_{\norm{\vec{w}}_2=1}\vec{w}^* {\mat{\hat K}} \vec{w} 
		\le \norm{{\mat{\hat K}}}_{\infty}
		\le \max_j \sum_{k=1}^{M} \abs{\dotprod{I_k}{I_j}_{L^2(\T^d)}}\\
		&\le  \frac{\lambda^d\beta^{d/2}}{N^d}\lambda^{2\lambda-2} \CC^2 \left(\lambda+\lambda 2^d(2^d-1) \sum_{m=1}^{\floor{N/2\rho}} m^{d-1} \max_{\vec{t}\in J_m} \frac{\sqrt{\beta}}{2}\frac{(\lambda\beta)^{\beta}}{(N\wrapm{\vec{t}})^{\beta}}\right) \\
		&\le \frac{\lambda^d\beta^{d/2}}{N^d}\lambda^{2\lambda-1}\CC^2 \left(1+\lambda^{\beta}\beta^{\beta+1/2} 2^{d-1}(2^d-1) \zeta(\beta-d+1)\rho^{-\beta} \right).
	\end{align*}
	Condition \eqref{eq:clusterSepThm} implies	
	\begin{equation*}
		\normltm{f}
		\le \sqrt{\frac{3}{2}}\left(\frac{\lambda\sqrt{\beta}}{N}\right)^{d/2} \lambda^{\lambda-1/2}\CC
	\end{equation*}
	and Lemma \ref{la:robustDual} finally the result.
\end{proof}

For $d =1$, Remark \ref{rm:nodeGeometry} iii) applied to the cluster complexity yields:
\begin{corollary}\label{cor:Ctau}
	Under the assumptions of Theorem \ref{thm:lowerBoundSminImproved} with $d=1$ and $\beta=2$,
	we have
	\begin{equation*}
		\smin(\mat{A}) 
		\ge \frac{1}{1.8(2\e)^{\lambda-1}}\cdot \sqrt{N} \tau^{\lambda-1}.
	\end{equation*}
\end{corollary}

\begin{example}[Specific choices of $\beta$]\label{ex:specificBeta}
 Specific choices of $\beta$ in Theorem \ref{thm:lowerBoundSminImproved} yield the following:
	\begin{enumerate}
	 \item By choosing $\beta=d+1$ or $\beta=d+2$ for $d$ being odd or even, respectively, and some additional cosmetics, the condition
	  \begin{equation*}
		\rho \ge 6 d \lambda \left({\lambda^{\lambda}\CC}\right)^{\frac{1}{d+1}}
	  \end{equation*}
	  implies \emph{our best estimate}
	  \begin{equation*}
		\smin(\mat{A})
		\ge\left(3 d^{d/4} \lambda^{\lambda+d/2-1/2}\right)^{-1} \frac{N^{d/2}}{\CC}.
	  \end{equation*}
	\item By choosing $\beta =2\ceil{\frac{1}{2}\log\left(2^{d}(2^d-1)\lambda^{\lambda} \zeta(2) \CC\right)}$ and noting that $\sqrt[2\beta]{\beta}\le 1.2$ for $\beta$ even and $\sqrt[\log C]{C}=\e$, 
	\emph{our weakest condition}
	  \begin{equation*}
		\rho 
		\ge 3.3 \lambda \left(2.5+1.4 d + \lambda\log\lambda+\log\CC\right),
	  \end{equation*}
	  implies
	  \begin{equation*}
		\smin(\mat{A}) 
		\ge  \left(1.5 \cdot \left(2.5+1.4 d + \lambda\log{\lambda} + \log \CC\right)^{d/4}\lambda^{\lambda+d/2-1/2}\right)^{-1} \frac{N^{d/2}}{\CC}.
	  \end{equation*}
	\end{enumerate}
\end{example}

\begin{example}[Well separated nodes]\label{ex:wellSeparated}
 For $\lambda=1$, we have $\CC=1$ and the nodes are well separated.
 For $\rho \ge 6d$, Example \ref{ex:specificBeta} i) yields
 \begin{equation*}
	\smin(\mat{A}) \ge \frac{N^{d/2}}{3d^{d/4}}.
 \end{equation*}
 Note that Theorem \ref{thm:lowerBoundSminImproved} always assumes $\rho\ge\beta\ge d+1$.
 This compares to \cite{KuMoPeOh17}, where $\rho\ge 3+2\log d$ already suffices for $\smin(\mat{A})>0$.
 Using Theorem \ref{thm:lowerBoundSminImproved} directly for $d=1$ and $\beta=2$, then $\rho \ge 4.4$ implies
 \begin{equation*}
   \smin(\mat{A}) \ge \frac{\sqrt{N}}{1.8}.
 \end{equation*}
This compares to \cite{AuBo17,Mo15}, which provide under the same condition on $\rho$, $\smin(\mat{A}) \ge \sqrt{N}\cdot\sqrt{1-1/\rho}\ge  \sqrt{N}/1.14$.
\end{example}

\begin{example}[Pair clusters]\label{ex:pairs}
 For $\lambda=2$, we have $\CC=1/\tau$ and at most pairs of nodes form clusters.
 Example \ref{ex:specificBeta} i) with 
 \begin{equation*}
  \rho \ge 12 d \left(\frac{4}{\tau}\right)^{\frac{1}{d+1}}
 \end{equation*}
 implies
  \begin{equation*}
    \smin(\mat{A})\ge\frac{\tau N^{d/2}}{12\cdot2^{d/2-1/2}\cdot d^{d/4}}.
  \end{equation*}
\end{example}

\begin{example}[Pair clusters, comparison]\label{ex:pairsComparison}
 Let $d=1$ and $\lambda=2$. We apply Theorem \ref{thm:lowerBoundSminImproved} with
 $\beta=2$, $\beta =2\ceil{\frac{1}{2}\log\left(\frac{\pi^2}{3} \lambda^{\lambda} \CC\right)}$ and $\beta=2\lambda$, respectively.
 These results are compared
 to \cite[Thm.~1]{LiLi17} (with minor corrections and where we simplified slightly $\floor{{n}/{\lambda}}\approx {n}/{\lambda}$),
 to \cite[Thm.~4.9]{KuNa18} (under the additional assumption that all nodes inside the clusters have the same separation), and
 to \cite[Cor.~4.2]{Di19} (with a minor improvement for $\tau \le 1$ and in estimating \cite[Eq.~(8)]{Di19}).
 \begin{table}[h!]
 \centering
  \begin{tabular}{|c|c|c|c|c|c|c|}
  \hline
   Ref. & \multicolumn{3}{c|}{Thm.~\ref{thm:lowerBoundSminImproved}} & \cite{LiLi17} & \cite{KuNa18} & \cite{Di19}\\[2ex]
   \hline
   $\rho\ge$ & $\frac{17.3}{\sqrt{\tau}}$ & $34.9+6.6\abs{\log\tau}$ & $\frac{29}{\sqrt[4]{\tau}}$ & $\frac{42.5\sqrt[4]{M}}{\sqrt[4]{\tau}}$ & $25(\log(\floor{\frac{M}{4}})+1)$ & $3$\\[2ex]
   \hline
   $\smin(\mat{A}) \ge$ & $\frac{\tau\sqrt{N}}{7.2}$ & $\frac{\tau \sqrt{N}}{6 \sqrt[4]{5.3 + \abs{\log\tau}}}$ & $\frac{\tau\sqrt{N}}{8.6}$ & $\frac{\tau\sqrt{N}}{4.5 \sqrt{M}}$
    & $\frac{\tau \sqrt{N}}{3.5}$ & $\frac{\tau\sqrt{N}}{1.7}$\\
    \hline
  \end{tabular}
 \end{table}
 These comparisons are also presented in section \ref{sec:pairClusters} numerically.
\end{example}

\begin{example}(Comparison with \cite{LiLi17})\label{ex:comparisonWithLiLi}
	Let $d=1$ and $\beta=2\lambda$, then $N > 2\lambda^3$ and
	 $\rho \ge 4.4\lambda^{5/2}\CC^{\frac{1}{2\lambda}}$
	 imply
	 \begin{equation*}
		\smin(\mat{A})
		\ge\left(1.8 \cdot C_0^{\lambda-1} \cdot \lambda^{\lambda+1/4}\right)^{-1} \frac{\sqrt{N}}{\CC},
	 \end{equation*}
	 where we set $C_0=1$ for the moment.
	  This can be compared to \cite[Thm.~1]{LiLi17}, where after minor corrections $N> 2\lambda^2$ and
	  $	\rho 
		\ge 10 \lambda^{5/2} (M\CC)^{\frac{1}{2\lambda}},
	  $ imply
	 \begin{equation*}
		\smin(\mat{A}) 
		\ge \left(1.5\cdot C_0^{\lambda-1} \cdot \sqrt{M} \lambda^{\lambda}\right)^{-1} \frac{\sqrt{N}}{\CC}.
	 \end{equation*}
	 According to Remark \ref{rm:LiLiEstimate}, $C_0 \in \intercr{\pi^{-1},1}$ depending on $\lambda$ and $n$.
	 In total, we have a stronger condition on $N$ but our condition on $\rho$ is always weaker and our estimate on $\smin(\mat{A})$ is sharper if $M>2$. 
	 This comparison is also presented in Figure \ref{fig:vglLiLi17}.
\end{example}

\begin{example}[All nodes cluster]
	Let $d=1$ and $\lambda=M$. If $N > 8M$, then Corollary \ref{cor:Ctau} implies
	\begin{equation*}
		\smin(\mat{A}) 
		\ge \frac{1}{1.8(2\e)^{M-1}}\cdot \sqrt{N} \tau^{M-1}.
	\end{equation*}
	This compares to \cite{BaDeGoYo18}, where the restriction of the nodes to an interval of length $1/(2M^2)$ and $N\ge 4M^3$ imply
	\begin{equation*}
		\smin(\mat{A}) 
		\ge \frac{1}{2^MM^{2M-1}}\cdot \sqrt{N} \tau^{M-1},
	\end{equation*}
	but, note that the definition of a clustered node configuration in \cite{BaDeGoYo18} is in principle more flexible than ours.
\end{example}

\section{Upper bounds and beyond distances} 
\label{sec:uBounds_beyondDistances}
In this section, we show that the obtained lower bounds are sharp for $d=1$ and for $\lambda=2$, respectively. 
Moreover, we show for $d>1$ and nodes in generic position (e.g.~not all nodes on a line for $d=2$), that the cluster complexity $C$ is not the optimal quantity to understand the situation here.
If we assume a normalized minimal separation $\tau$ between nodes, then the estimate in Theorem \ref{thm:lowerBoundSminImproved} is sub-optimal with respect to the order in $\tau$ we can derive from the cluster complexity.
For this, we give an example with one cluster of three nodes in the bivariate case, $d=2$.

\begin{example}[Matching bounds for $d=1$]
 	In the second version of \cite[Prop.~3]{LiLi17} an upper bound on $\smin(\mat{A})$ is given for a clustered node configuration that consists of at least one cluster of $\lambda$ equispaced, $\tau$ separated nodes. After further simplifications, we can derive\\
 	\begin{equation*}
	 	\min_{\Omega}\smin(\mat{A}) 
	 	\le (\pi\lambda)^{1/4} \pi^{\lambda-1} \sqrt{N} \tau^{\lambda-1}(1+\tau C(\lambda)\sqrt{N}), \quad C(\lambda):= 2\pi \sum_{l=0}^{\lambda} \nchoosek{\lambda-1}{l}\frac{l^\lambda}{\lambda!}.
 	\end{equation*}
 	Together with Remark \ref{rm:LiLiEstimate} and Corollary \ref{cor:Ctau} this assures that for sufficiently large $N\in\N$, small $\tau$ and $\lambda\ge 2$, there exist constants $ c_1\le c_2$ such that
 	\begin{equation*}
 	  \sqrt{N} \left(c_1 \tau\right)^{\lambda-1} \le \min_{\Omega}\smin(\mat{A}) \le \sqrt{N} \left(c_2 \tau\right)^{\lambda-1},
 	\end{equation*}
 	where the minimum is taken over all clustered node configurations $\Omega$ with at least one cluster of $\lambda$ nodes with normalized minimal separation $\tau$.
 	
	This was also expected in \cite[Rem.~3.5]{BaDeGoYo18}. In particular note that the lower bound in Remark \ref{rm:nodeGeometry} iii) implies that the term $\lambda^{\lambda}$ in Theorem \ref{thm:lowerBoundSminImproved} cannot be avoided.
\end{example}

\begin{example}[Matching bounds for $\lambda=2$]\label{ex:pairCluster_multivariate}
 Let $d\in\N$, $\lambda\ge 2$, and $\tau\le1$ be such that $\wrapm{\vec{t}_1-\vec{t}_2}=\tau/N$, then the Cauchy interlacing theorem for eigenvalues (\cite[Thm.~4.3.28]{HoJo13}) and the binomial formula yield
 \begin{align*}
  \smin(\mat{A})^2
  &\le N^d\left(1 - \abs{d_n\left(\vec{t}_1-\vec{t}_2\right)}\right)
  \le N^d\left(1 - \left(1-\frac{\pi^2\tau^2}{6}\right)^d\right)\\
  &\le N^d \frac{\pi^2\tau^2}{6} \sum_{k=0}^{d-1} \left(1-\frac{\pi^2\tau^2}{6}\right)^k
  \le \frac{\pi^2   \tau^2 dN^d}{6}.
 \end{align*}
 Together with Example \ref{ex:pairs}, there exists constants $c_1(d)\le c_2(d)$ such that
 \begin{equation*}
  c_1(d) N^{d/2}\tau\le\min_{\Omega} \smin(\mat{A})\le c_2(d) N^{d/2} \tau,
 \end{equation*}
 where the minimum is taken over all clustered node configurations $\Omega$ with at least one cluster of $\lambda=2$ nodes with normalized minimal separation $\tau$.
\end{example}

\begin{example}[Triple cluster]\label{ex:tripleCluster}
	Let $d=2$, $N\in\N$, and $\Omega=\set{\vec{t}_1, \vec{t}_2, \vec{t}_3}\subset \intercl{0,1}^2$ with
	\begin{equation*}
	 \vec{t}_1=\pmat{0\\0},\; \vec{t}_2=\frac{\nu}{N} \pmat{a_1\\a_2},\; \vec{t}_3=\frac{\nu}{N} \pmat{b_1\\b_2},\quad
	 \nu\in\left(0,\frac{1}{2}\right],\;a_1^2+a_2^2=b_1^2+b_2^2=1,\; a_1b_1+a_2b_2\le 0.
	\end{equation*}
	and hence, the normalized minimal separation of $\Omega$ is $\nu/\sqrt{2}\le \tau \le \nu$.
        Then the smallest singular value of the corresponding Vandermonde matrix $\mat{A}$ fulfills
	\begin{equation*}
	 \smin(\mat{A}) = \begin{cases}
	                   \Theta(\tau^2), & \text{antipodal nodes},\; a_1 b_1 + a_2 b_2=-1,\\
	                   \Theta(\tau), & \text{otherwise},
	                  \end{cases}
	\end{equation*}
	and this can be seen as follows:
	Define the real matrix
	\begin{equation*}
		\mat{M} :=\pmat{1 & u & v \\ u & 1 & w \\ v & w &1}
		:= \left(\e^{-\pi\ii m (\vec{t}_j-\vec{t}_k)}d_m(\vec{t}_j-\vec{t}_k)\right)_{j,k=1,2,3},
	\end{equation*}
	note that $\smin(\mat{A})^2 =\smin(\mat{A}\mat{A}^*)=\smin(\mat{M})= \norm{\mat{M}^{-1}}_2^{-1}$, and use the explicit formula
	\begin{equation}\label{eq:KInvFormula}
		\norm{\mat{M}^{-1}}_2 
		= \frac{1}{\abs{u^2+v^2+w^2 -2uvw -1}} \norm{\pmat{w^2-1 & u-vw & v-uw \\ u-vw & v^2-1 & w-uv\\ v-uw & w-uv & u^2-1}}_2.
	\end{equation}
	The univariate Taylor expansion
	\begin{equation*}
	 \e^{-\pi\ii n \nu/N} d_n\left(\frac{\nu}{N}\right)= 1- \alpha_n \nu^2 + \gamma_n \nu^4 + \OO{\nu^6},\quad \alpha_n, \gamma_n \ne 0,
	\end{equation*}
	and $a_1^2+a_2^2=1=b_1^2+b_2^2$ yield
	\begin{equation*}
	  u = 1-\alpha_n \nu^2 + ( \alpha_n^2 a_1^2a_2^2  + \gamma_n (a_1^4+a_2^4) ) \nu^4 +\OO{\nu^6}
	\end{equation*}
	and similar expressions for the other quantities.
	By direct computation, we see that the entries in the matrix on the right hand side of \eqref{eq:KInvFormula} are all $\OO{\nu^2}$ and for example the diagonal entry $u^2-1$ is $\Theta(\nu^2)$ independent of $\vec{a}$ and $\vec{b}$.
	Hence, the norm of that matrix is $\Theta(\nu^2)$.
	Similarly, the denominator of \eqref{eq:KInvFormula} can be computed to be
	\begin{equation*}
		u^2+v^2+w^2 -2uvw -1 =
		\begin{cases}
			\OO{\nu^6}, & a_1 b_1 + a_2 b_2=-1,\\
			\Theta(\nu^4), & \text{otherwise}.
		\end{cases}
	\end{equation*}
	Finally, this yields
	\begin{equation*}
		\smin(\mat{A})=
		\begin{cases}
			\OO{\nu^2}, & a_1 b_1 + a_2 b_2=-1,\\
			\Theta(\nu), & \text{otherwise},
		\end{cases}
	\end{equation*}
	and together with Theorem \ref{thm:lowerBoundSminImproved} the assertion.
\end{example}

\section{Numerics}

 In this section we do four different experiments. Two of them are to compare our results with recent results from the literature ($d=1$) and two of them underline our results from section \ref{sec:uBounds_beyondDistances}. All computations were carried out using MATLAB R2017b.

%

 \subsection{Pair clusters}\label{sec:pairClusters}

In order to compare our results (see Example \ref{ex:pairsComparison}) with the ones from the second version of \cite[Thm.~2]{LiLi17}, \cite{Di19} and \cite{KuNa18}, we set $d=1$, $N=2^{15}+1$ (\cite{KuNa18} requires odd $N$ without further considerations), and take $M=4$ and $M=20$ nodes, respectively. The node configuration consists of uniformly placed clusters (at $l/N$, $l=1,\dots,M/2$) that include two nodes each. The first cluster realizes the minimal separation $\tau$, which is picked logarithmically uniformly at random from $[10^{-12},1]$, i.e.~$t_1=0$ and $t_2=\tau/N$. The further clusters have nodes $t_{2l}=l/N$ and $t_{2l+1}=(l+\delta)/N$ for $l=1,\dots,(M-1)/2$, where $\delta\in [\tau,2\tau]$ (parameter $c=2$ in \cite[Thm.~4.7]{KuNa18}) is picked uniformly randomly. Afterwards, we compute $\smin(\mat{A})$, where $\mat{A}$ is the Vandermonde matrix defined in \eqref{eq:defA} corresponding to the node configuration. For each $M$ we pick $50$ instances of $\tau$ and the results are presented in Figure \ref{fig:vglLiLi17_Di19_KuNa18}. 
\begin{figure}[h]
	\begin{center}
		\includegraphics[width=0.8\textwidth]{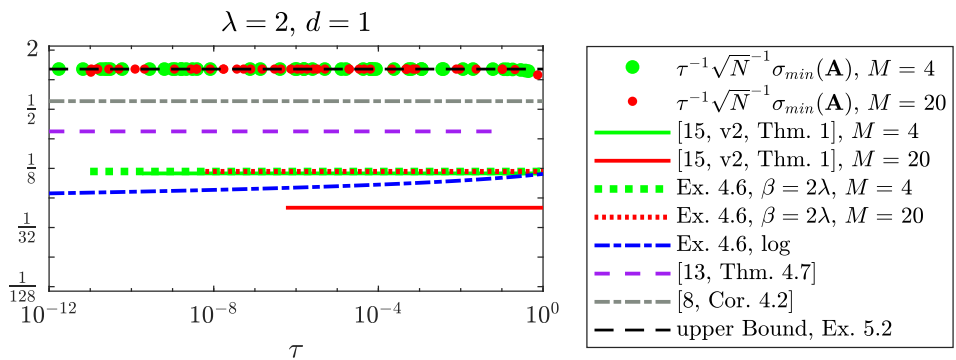}
	\end{center}
	\caption{Comparison of different results for the case of pair clusters as in Example \ref{ex:pairsComparison}.}
	\label{fig:vglLiLi17_Di19_KuNa18}
\end{figure}
This clustered node configuration fulfills $\rho\ge \frac{2N}{M}-1$ independently of $\tau$. Theorem \ref{thm:lowerBoundSminImproved} and the second version of \cite[Thm.~1]{LiLi17} make restrictions to $\tau$ through the condition on $\rho$. Therefore, choosing $\beta$ logarithmically as in Example \ref{ex:specificBeta} ii) requires $\tau \ge \e^{-\frac{35.9-2N/M}{6.6}}$,
which is below $10^{-200}$ for both $M=4$ and $M=20$. The second version of \cite{LiLi17} and our result with $\beta=4$ requires respectively\setlength{\columnsep}{-1cm}\vspace{-0.8cm}
\begin{multicols}{2}
\begin{equation*}
	\tau \ge \frac{43^4 M}{\rho^4}
	\approx \begin{cases}
		1.9\cdot 10^{-10},\quad &M=4,\\
		5.9 \cdot 10^{-7},\quad &M=20,
\end{cases}
\end{equation*}
\break
\begin{equation*}
	\tau \ge \frac{29^4}{\rho^4}
	\approx \begin{cases}
		9.8\cdot 10^{-12},\quad &M=4,\\
		6.1 \cdot 10^{-9},\quad &M=20.\\
	\end{cases}
\end{equation*}
\end{multicols}

\subsection{Bigger clusters}

In this numerical example, we confirm our results in the univariate case, $d=1$, for bigger clusters of size $\lambda = 5$ and compare them with the results from the second version of \cite{LiLi17}. The polynomial degree is set to $N=2^{15}$. We build up clustered node configurations with $L=2$ $(M=10)$ and $L=10$ $(M=50)$ clusters placed equispaced at $\frac{l}{L}$ for $l=0,\dots,L-1$. At each cluster position the cluster nodes start to lie equispaced with separation $\frac{\tau}{N}$, where $\tau \in [10^{-4},1/4]$ (the right hand interval bound is due to cluster lying in an interval of length $1/N$) is picked logarithmically uniformly at random. Afterwards the smallest singular value $\smin(\mat{A})$ is computed. This procedure is repeated 100 times for the respective choice of $L$ and the results are presented in Figure \ref{fig:vglLiLi17}. We use the statements from Example \ref{ex:comparisonWithLiLi} with $C_0 = (1-\frac{\pi^2}{3\lambda^2})^{-1/2}N/\lambda \floor{n/\lambda}^{-1}$. Since $d=1$, the worst case cluster complexity is estimated by \eqref{eq:CupperBound} to $\CC\le \tau^{-4}/4$.

\begin{figure}[h!]
	\centering
		\includegraphics[width=0.8\textwidth]{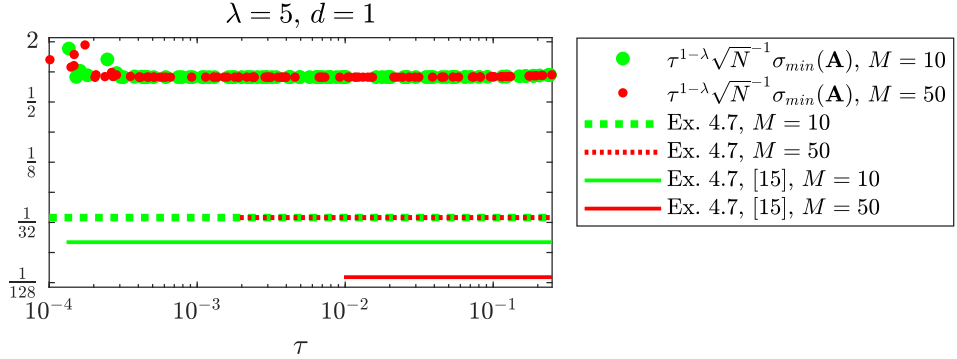}
	\caption{Node configurations with bigger clusters. Lower bounds on $\smin(\mat{A})$. Comparison with \cite[v2, Thm.~1]{LiLi17} as presented in Example \ref{ex:comparisonWithLiLi} with estimate from Remark \ref{rm:LiLiEstimate}.}
	\label{fig:vglLiLi17}
\end{figure}

\subsection{Pair clusters, bivariate}

We present a numerical experiment in order to confirm our results for the higher dimensional case and set $d=2$. Randomized clustered node configurations of $L=2$, $L=20$ and $L=40$ clusters with $2$ nodes each are constructed for $100$ different minimal separations $\tau$, respectively. Then the smallest singular values of the corresponding Vandermonde matrices $\smin(\mat{A})$ are computed and the upper bound from Example \ref{ex:pairCluster_multivariate} and the lower bound from Example \ref{ex:pairs} are shown. The results are presented in Figure \ref{fig:d2Lambda2}.
The node configurations are built as follows. The minimal separation $\tau$ is picked logarithmically uniformly at random in $[10^{-3},1]$. We set $N=10^3$ so that the condition on $\rho$ in Example \ref{ex:pairs} together with the left interval bound for $\tau$ make $\rho \ge \rho_{min}$ (value shown in the figure) necessary. Two clusters realize the cluster separation $\rho_{min}$ and for the remaining clusters, we pick a position in $[0,1]^2$ uniformly at random. The positions are fixed for the respective choice of $L$ and do not change for different $\tau$. Each cluster is constructed randomly by setting one node to $(0,0)$ and one to either $(a,1)$ or $(1,a)$ for some $a\in \interc{0,1}$. Then we scale the clusters by $\tau$ and move them to their respective cluster positions. 
\begin{figure}[h!]
	\centering
		\includegraphics[width=0.8\textwidth]{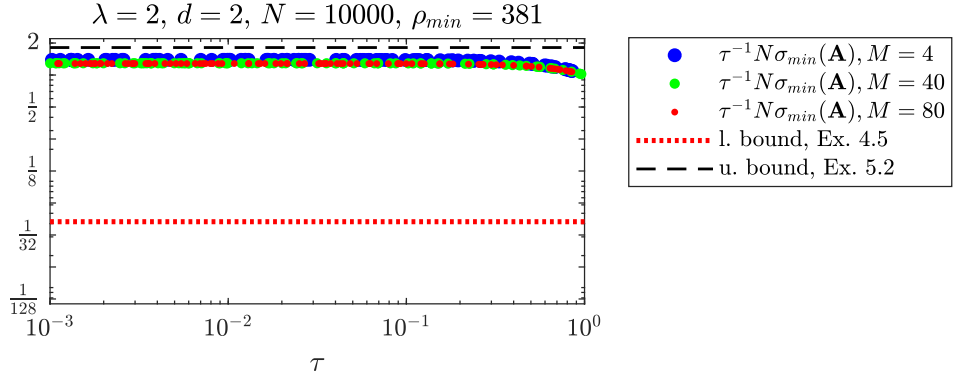}
	\caption{Upper and lower bounds on $\smin(\mat{A})$ for bivariate pair clusters as in Examples \ref{ex:pairs} and \ref{ex:pairCluster_multivariate}.}
	\label{fig:d2Lambda2}
\end{figure}

\subsection{One triple cluster, bivariate}
Here we present a numerical experiment for Example \ref{ex:tripleCluster}.
We set $N=100$, $d=2$ and build the triple cluster consisting of the nodes $\vec{t}_1=(0,0)^T$, $\vec{t}_2=(-\sqrt{1-a^2}\nu/N,a\nu/N)^T$ and $\vec{t}_3=(\nu/N,0)^T$ (see Figure \ref{fig:tripleCluster}, left), where $\tau=\nu\sqrt{1-a^2} \in [10^{-6},1/2]$ is picked logarithmically uniformly at random. Then we compute the smallest singular value of the Vandermonde matrix $\smin(\mat{A})$. This is repeated $100$ times for $a=0.1$ and $a=0$ each. The results are presented in Figure \ref{fig:tripleCluster} (right). We see the asymptotic behavior with respect to $\tau$ calculated in Example \ref{ex:tripleCluster}. Furthermore, for nodes not being antipodal, we observe that the asymptotic starts when $\tau$ becomes smaller than the displacement parameter $a$.
\begin{figure}[h!]
 \centering
 \begin{subfigure}{0.49\textwidth}
 \hfill
		\includegraphics[width=0.8\linewidth]{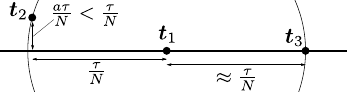}
	\label{fig:tripleCluster_antipodal}
 \end{subfigure}	
 \hfill
 \begin{subfigure}{0.49\textwidth}
	\includegraphics[width=0.8\linewidth]{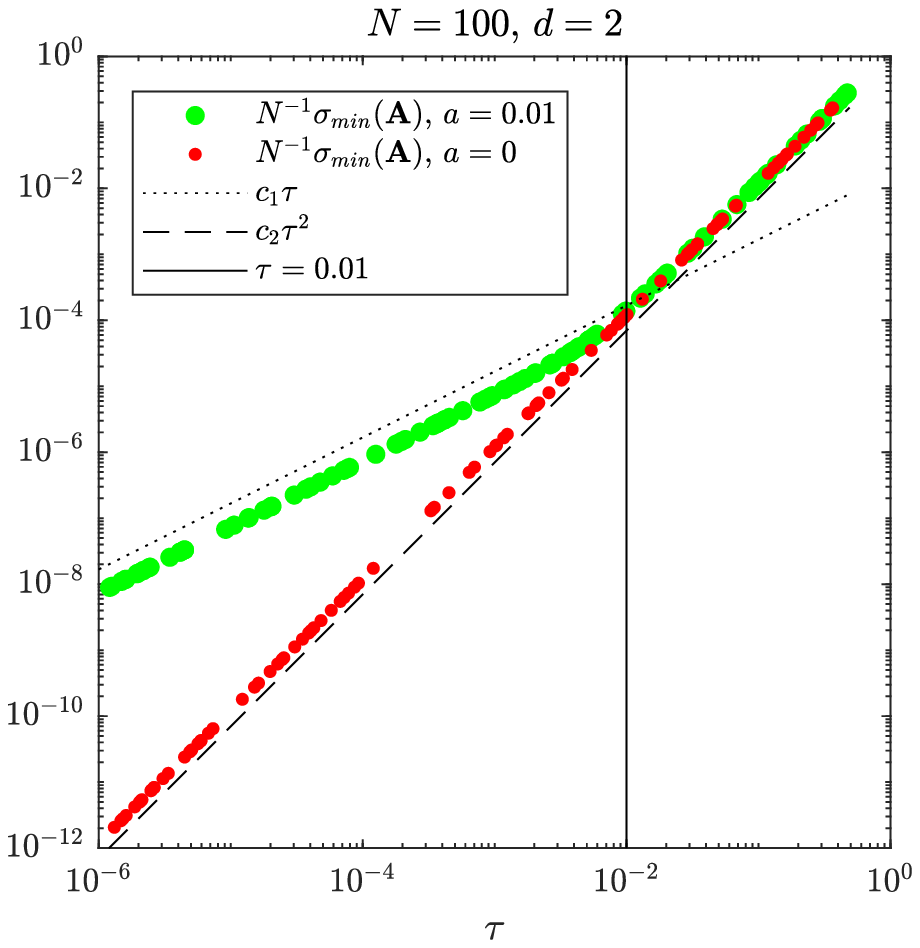}
 \end{subfigure}
 	\caption{Triple cluster, almost antipodal nodes, cf.~Example \ref{ex:tripleCluster}. Left: sketch of node positions. Right: numerical results.}
 \label{fig:tripleCluster}
\end{figure}


\textbf{Acknowledgements.} The authors thank J\"urgen Prestin for discussions on Lemma \ref{la:modDirich} and gratefully acknowledge support by the projects DFG-GK1916 and DFG-SFB944.

%

\bibliographystyle{abbrv}
\bibliography{references}

\end{document}